\documentclass{amsart}
\usepackage{cite,amsmath,amssymb,amsthm,amsrefs,hyperref}
 \newtheorem{thm}{Theorem}[section]
 \newtheorem{cor}[thm]{Corollary}
 \newtheorem{lem}[thm]{Lemma}
 \newtheorem{prop}[thm]{Proposition}
 \theoremstyle{definition}
 \newtheorem{defn}[thm]{Definition}
 \theoremstyle{remark}
 \newtheorem{rem}[thm]{Remark}
 \newtheorem{ex}[thm]{Example}
 \numberwithin{equation}{section}

\begin{document}

\title{Projective Cross-ratio on Hypercomplex Numbers}

\author{Sky Brewer}

\address{School Of Mathematics \\ Leeds University\\ Leeds \\ LS2 9JT
  \\England}

\email{jacomago at hotmail.com}



\date{December 16, 2011}


\begin{abstract}
  The paper presents a new cross-ratio of hypercomplex numbers based
  on projective geometry. We discuss the essential properties of the
  projective cross-ratio, notably its invariance under M\"obius
  transformations. Applications to the geometry of conic sections and
  M\"obius-invariant metrics on the upper half-plane are also given.

\end{abstract}

\maketitle

The cross-ratio is a number associated with four points, which is invariant
under M\"obius transformations. It has received some attention recently in the context
of quarternions \cite{Gwynne2011}. It is not well defined on hypercomplex numbers
due to the presence of zero divisors, hence we introduce the projective cross-ratio to
suit this situation.

\section{Preliminaries}
Firstly we look at definitions of hypercomplex numbers and their properties.

\subsection{Hypercomplex numbers}

Up to isomorphism, there are only three 2-dimensional
commutative algebras with a unit over the real numbers. Each is isomorphic to
one of the following hypercomplex numbers.

\begin{defn} Define
  \begin{itemize}
  \item complex numbers $\mathbb C = \{a+bi:a,b \in \mathbb R, i^2 =-1
    \};$
  \item dual numbers $\mathbb D = \{ a+b\epsilon : a,b \in \mathbb R ,
    \epsilon^2 = 0 \};$
  \item double numbers $\mathbb O = \{ a+bj: a,b \in \mathbb R, j^2 =
    1 \}.$
  \end{itemize}
  We use the notation $\mathbb A = \{a+b
  \iota:a,b \in \mathbb R \}$ to represent any of the above.
\end{defn}

An important quantity is the respective modulus.

\begin{defn}
  Let $z = x+y \iota \in \mathbb A$, we define the hypercomplex conjugate as $\bar z = x-y \iota$. Also the hypercomplex modulus is defined as $|z|^2 = z \bar z$.
\end{defn}

\begin{defn}
  $a \in \mathbb A$ is a zero divisor or $a \mid 0$ if and only if  $a \neq 0$ and there exists $b \in \mathbb A$ such that $ab
  =0$.
\end{defn}

\begin{rem}
  $z \in \mathbb A$ is a zero divisor, if and only if
  $|z|^2 = z \bar z =0$. As the set of zero divisors for $ \mathbb C,
  \mathbb D$ and $\mathbb O$ are $\emptyset, \{r\epsilon ; r \in \mathbb R
  \}$ and $\{ r(1+j); r \in \mathbb R \} \cup \{ r(1-j); r \in \mathbb
  R \}$ respectively.
\end{rem}

\subsection{Projective geometry}

The projective cross-ratio is defined on a projective space.

\begin{defn}
  $\lambda \in \mathbb A$ is unit if it has a multiplicative inverse.
\end{defn}

\begin{defn}
  Given $x,y,u,v \in \mathbb A$ and not both $x,y$ or $u,v$ are zero. Then $(x,y) \sim (u,v)$ if and only if there exists a unit $\lambda \in \mathbb A$ such that $(x,y) = (\lambda u, \lambda v)$.
\end{defn}
The hypercomplex projective space is defined as follows.

\begin{defn}
  The projective space is the set of equivalence classes of the equivalence relation $\sim$,
  \begin{equation}
    \mathbb P^1(\mathbb A) = \{ [x,y] : x,y \in \mathbb A \mbox{ and } x \sim y  \}.
  \end{equation}
\end{defn}
The following is a definition of mapping from $\mathbb A$ to
$\mathbb P^1(\mathbb A)$ and the other way round.

\begin{defn}
  The map $\mathfrak S:\mathbb A \to \mathbb P^1(\mathbb A)$ is defined by
  $\mathfrak S(z) = [z,1]$. Also the map $\mathfrak P:\{ [x,y] : y \mbox{ does
    not divide } 0 \} \mapsto \mathbb A$ is defined as $\mathfrak
  P((x,y)) = x/y$.
\end{defn}

Two numbers $z$ and $w$ are ``distinct'' if $(z-w) \neq 0$.

\begin{defn}
  Two points \((x_1,y_1), (x_2,y_2) \in \mathbb P^1(\mathbb A) \) are
  essentially distinct if \(x_1y_2-y_1x_2\) is not a zero divisor or zero.
\end{defn}

Define $\underline \infty, \underline 1$ and $\underline 0 \in \mathbb P^1(\mathbb A)$
as notation for the equivalence classes $\begin{bmatrix}
  1 \\
  0
\end{bmatrix}$, $\begin{bmatrix}
  1 \\
  1
\end{bmatrix}$ and $\begin{bmatrix}
  0 \\
  1
\end{bmatrix}$ respectively. In particular, for all 3 types of hypercomplex numbers
$\underline \infty = [z,0]$ such that $z$ is not a zero divisor.

\begin{rem}
Given $\underline z \in \mathfrak S(\mathbb A)$:
\begin{enumerate}
\item $|\mathfrak P(\underline z)| = |x| / |y|$,
\item $\overline {\mathfrak P(\underline z)} = \bar x / \bar y$.
\end{enumerate}
\end{rem}

Hence the following definition:
\begin{defn} Given $\underline z =(x,y) \in \mathbb P^1(\mathbb A)$
  \begin{enumerate}
  \item Define conjugate on projective numbers by $\underline {\bar z}
    = [\bar x,\bar y]$
  \item Define modulus on projective numbers by $|\underline z|^2 =
    (|x|^2,|y|^2)$.
  \end{enumerate}
\end{defn}

\subsection{M\"obius Transformations}

M\"obius transformations will play an important role in this paper.
\begin{defn}
  A M\"obius transformation is a function $f$, of hypercomplex
  variables $z \in \mathbb A$. It can be written in the form
  \begin{equation}
    f(z) = \frac{az+b}{cz+d},
  \end{equation} for some $a,b,c,d \in \mathbb A$, such that $ad-bc$ is a unit.
\end{defn}

\begin{defn}
  The General linear group is defined as $GL_2(\mathbb A) = \{ A:det(A) \mbox{ a unit} \}$. The Projective linear group is the quotient group $PGL_2(\mathbb A)
  = GL_2(\mathbb A)/ \{\lambda I: \lambda \mbox{ a unit} \}$

  Define a mapping by a matrix $A \in PGL_2(\mathbb A)$ on $\underline
  z \in \mathbb P^1(\mathbb A)$ by the usual vector multiplication.
\end{defn}

\begin{lem}
  $PGL_2(\mathbb A)$ action on $\mathbb P^1(\mathbb
  A)$ is well defined.
\end{lem}

\begin{proof}
  True by linearity of $\sim$.
\end{proof}

There is a group homomorphism from the group of M\"obius transformations
acting on $\mathbb A$ to the group $PGL_2(\mathbb A)$ acting on
$\mathfrak S(\mathbb A)$.

\begin{rem}
  Given a M\"obius transformation $f(z)$, $x,y \in \mathbb A, y \nmid
  0$.
  \begin{align*}
    \mathfrak S(f(x/y)) & = \mathfrak S\left(\frac{a(x/y)+b}{c(x/y)+d}\right) \\
    & =  \mathfrak S\left(\frac{ax+by}{cx+dy}\right) \\
    & = \begin{bmatrix}
      ax +by \\
      cx +dy
    \end{bmatrix} \\
    & = \begin{pmatrix}
      a & b \\
      c & d
    \end{pmatrix}
    \begin{bmatrix}
      x \\
      y
    \end{bmatrix}.
  \end{align*}
  Hence the representation of M\"obius transformation by
  $PGL_2(\mathbb A)$.
\end{rem}

\begin{defn}\label{sec:projective-geometry:invariantset}From Yaglom's book \cite{Yaglom1979}*{p277} we define a subset of $\mathbb P^1(\mathbb A)$ as:
  \begin{equation}
   S = \{A \underline z: A \in GL_2(\mathbb A),\underline z \in \mathfrak S(\mathbb A) \} /\{\lambda : \lambda \mbox{ a unit} \}.
  \end{equation}
\end{defn}

Notice that $S$ is invariant under $PGL_2(\mathbb A)$.
\begin{lem}\label{sec:projective-geometry:infinitysets}
  $S$ consists of the union of $\mathfrak S(\mathbb A)$, $
  \{ \underline \infty \}$ and one of the following:
  \begin{enumerate}
  \item for $\mathbb C$: the empty set,
  \item for $\mathbb D$: the set
    $
      \left\{
      \begin{bmatrix}
        t \\
        \epsilon
      \end{bmatrix}, t \in \mathbb R \right\},$
  \item for $\mathbb O$: the set \begin{align*}\left\{
      \begin{bmatrix}
        t \\
        1 + j
      \end{bmatrix}, t \in \mathbb R \right\} & \cup  \left\{
      \begin{bmatrix}
        t(1-j) \\
        1+j
      \end{bmatrix},  t \in \mathbb R\right\} \cup \\
    & \qquad \left\{
      \begin{bmatrix}
        t \\
        1 -j
      \end{bmatrix}, t \in \mathbb R \right\}\cup \left\{
      \begin{bmatrix}
        t(1+j) \\
        1-j
      \end{bmatrix}, t \in \mathbb R \right\}
 .
    \end{align*}
  \end{enumerate}
\end{lem}
\section{Projective Cross ratio}
\subsection{M\"obius transformations on $\mathbb P^1(\mathbb A)$}

The paper follows Alan Beardon's book ~\cite{Beardon2005}*{Chapter
  13}. Differences are:
\begin{itemize}
\item In \cite{Beardon2005}*{Thm 13.2.1}, ``distinct'' has been replaced by
``essentially distinct''.
\item M\"obius maps have been replaced by matrices
from $PGL_2(\mathbb A)$.
 \item Complex numbers have been replaced by members
of $\mathbb A$ for use with projective points.
\end{itemize}
The following theorem relates to \cite{Beardon2005}*{Thm 13.2.1} and includes the definition of an important matrix.

\begin{thm}
  Given two sets of pairwise essentially distinct points $\{
  \underline z_1, \underline z_2, \underline z_3 \}$ and $\{
  \underline w_1, \underline w_2, \underline w_3\}$ from $\mathbb
  P^1(\mathbb A)$, there exists a unique matrix $A \in PGL_2(\mathbb
  A)$, such that $A \underline z_i = \underline w_i$ for all $i
  =1,2,3$.
\end{thm}

\begin{proof}
  Let $\underline z_i = [x_i, y_i], \underline w_j = [u_j, v_j] \in
  \mathbb P^1(\mathbb A)$ and define $A \in PGL_2(\mathbb A)$ as
  \begin{equation}\label{eq:1}
    A =
    \begin{pmatrix}
      (x_2y_3 - x_3y_2)y_1 & -(x_2y_3 - x_3y_2)x_1 \\
      (x_2y_1 - x_1y_2)y_3 & -(x_2y_1 - x_1y_2)x_3
    \end{pmatrix} \in PGL_2(\mathbb A).
  \end{equation} The points are essentially distinct implies $\det(A)=(x_2y_1 - x_1y_2)(x_1y_3 - x_3y_1)(x_2y_3 - x_3y_2)$ does not divide zero, hence $A^{-1}$ exists.
  A direct calculation shows $A \underline z_1 =
  \underline 0,
  A \underline z_2 =
  \underline 1,
  A \underline z_3 =
  \underline \infty.$ Define $B = A'^{-1}$ where $A'$ is the same as $A$ with $x_i=u_i$  and $y_i = v_i$. Then define $M=BA$. $M$ is the required matrix.

  To prove the uniqueness, suppose that $M, N \in PGL_2(\mathbb A)$ are distinct matrices,
  such that $M \underline z_i =
  \underline w_i = N \underline z_i$ for $i=1,2,3$. Then
  $M^{-1}N$ fixes each $\underline z_i$. Let $V$ be the matrix that
  maps $\underline z_1,\underline z_2,\underline z_3$ to $\underline
  0, \underline 1, \underline \infty$ respectively. Then $A =
  V^{-1}M^{-1}NV$ is a matrix that fixes $\underline 0,\underline 1,
  \underline \infty$, this then must be the identity matrix. We can check this
  by solving the set of linear equations $A \underline z =
  \underline z$ for $\underline z = \underline 0,\underline 1,
  \underline \infty$. Hence as $PGL_2(\mathbb A)$ is a group we have
  that $M =N$.
\end{proof}

\begin{cor}
  If a matrix $A$ fixes three pairwise essentially distinct points, then $A =
  I$.
\end{cor}

\begin{proof}
  The theorem states that there exists a unique matrix which satisfies the
  above property. The identity matrix satisfies the property, so it is the
  unique matrix.
\end{proof}

\subsection{Cross-Ratio}

\begin{defn}\label{crossratio}
  The original cross-ratio from ~\cite{Beardon2005}*{p261} of four pairwise essentially
  distinct points $z_1,z_2, z_3, z_4$ in $\mathbb A$ is defined as
  \begin{equation}
    [z_1,z_2,z_3,z_4] = \frac{(z_1 -z_3)(z_2-z_4)}{(z_1-z_2)(z_3 - z_4)}.
  \end{equation}
\end{defn}

The Projective cross-ratio is constructed from the entries from the matrix ~\eqref{eq:1}.
The properties from the original cross-ratio are replicated on the projective space.

\begin{defn}
  The projective cross-ratio of four distinct points $
  \underline z_1,\underline z_2, \underline z_3, \underline z_4 \in
  \mathbb P^1(\mathbb A)$ such that $\underline z_i =
  \begin{bmatrix}
    x_i \\
    y_i
  \end{bmatrix}$, for $ i \in \{1,2,3,4\}$ is defined as:
  \begin{equation}
    [ \underline z_1,\underline z_2,\underline z_3,\underline z_4 ] =
    \begin{bmatrix}
      (x_1y_3 - x_3y_1)(x_2y_4 - x_4y_2) \\
      (x_1y_2 - x_2y_1)(x_3y_4 - x_4y_3)
    \end{bmatrix} \in \mathbb A^2.
  \end{equation}
  Four points are singular if their projective cross-ratio is not
  in $S$ from Definition \ref{sec:projective-geometry:invariantset}.
\end{defn}

Note that for complex numbers the projective cross-ratio gives the
same result:
\begin{rem}\label{sec:cross-ratio:rem:4}
  For any pairwise distinct $z_1,z_2,z_3,z_4 \in \mathbb C$:
  \begin{equation}[z_1,z_2,z_3,z_4] = \mathfrak P([\mathfrak
    S(z_1),\mathfrak S(z_1),\mathfrak S(z_1),\mathfrak
    S(z_1)]) \end{equation} where the left-hand side contains the
  original cross-ratio and the right-hand side, the projective one.
\end{rem}

The following Lemmas
link $\underline 0, \underline 1$ and $\underline \infty \in \mathbb P^1(\mathbb A)$
to the projective cross-ratio.
\begin{lem}Let $\underline z \in \mathbb P^1(\mathbb A)$, then $ [
  \underline 0, \underline 1, \underline z, \underline \infty] =
  \underline z$
\end{lem}

\begin{proof}
  $ [\underline 0, \underline 1, \underline z, \underline \infty ]
  =
  \begin{bmatrix}
    (0 - x)(0-1) \\
    (0 - 1)(0 - y)
  \end{bmatrix} = \underline z $
\end{proof}

\begin{lem}\label{rem 1}
  Given the matrix $A \in PGL_2(\mathbb A)$ which maps $\underline
  z_1, \underline z_2, \underline z_4 \in \mathbb P^1(\mathbb A)$ to $\underline 0, \underline 1,
  \underline \infty$ respectively, then $A \underline z = [\underline
  z_1,\underline z_2,\underline z,\underline z_4]$
\end{lem}

\begin{proof}
  Calculate $A \underline z$.
\end{proof}

The next Theorem corresponds to \cite{Beardon2005}*{Thm 13.4.2}. It shows
a necessary and sufficient condition for an existence of a matrix $A
\in PGL_2(\mathbb A)$ such that $A$ maps four essentially
distinct projective points to another four.
\begin{thm}
  Given two sets of pairwise essentially distinct points $\underline
  z_i,\underline w_i \in \mathbb P^1(\mathbb A)$, for $i =
  1,2,3,4$. We have the equality
  $[\underline z_1,\underline z_2,\underline z_3,\underline z_4] =
  [\underline w_1,\underline w_2,\underline w_3, \underline w_4]$ if
  and only if there exists $A \in PGL_2(\mathbb A)$ such that
  $A\underline z_i = \underline w_i$.
\end{thm}

\begin{proof}
  For sufficiency say $A$ is the required matrix and $A
  =
  \begin{pmatrix}
    a & b \\
    c & d
  \end{pmatrix}
  ,\underline z_i =
  \begin{bmatrix}
    x_i \\
    y_i
  \end{bmatrix}, \underline w_i =
  \begin{bmatrix}
    u_i \\
    v_i
  \end{bmatrix},A \underline z_i = \underline w_i$. Doing a substitution gives:
  \begin{align*}
    u_jv_i - u_iv_j & = (ax_j + b y_j)(cx_i + dy_i) - (ax_i + b
    y_i)(cx_j + dy_j) \\ & = (ad - bc)(x_jy_i - x_iy_j).
  \end{align*}
  Substitute this into the equation for the projective cross-ratio
  to give the required equivalence.

   For necessity say $[\underline z_1,\underline
  z_2,\underline z_3,\underline z_4] = [\underline w_1,\underline
  w_2,\underline w_3, \underline w_4]$. Let $H, G \in PGL_2(\mathbb
  A)$ such that $G\underline z_1 = \underline 0, G\underline z_2
  = \underline 1, G\underline z_4 = \underline \infty,
  H\underline w_1 = \underline 0, H\underline w_2 = \underline
  1, H\underline w_4 = \underline \infty$. It then follows:

  \begin{align*}
    G \underline z_3 & = [ \underline 0, \underline 1, G\underline
    z_3,
    \underline \infty ] \\
    & = [G\underline z_1,G\underline z_2,G\underline z_3,G\underline z_4] \\
    & = [\underline z_1,\underline z_2,\underline z_3,\underline z_4] \\
    & = [\underline w_1,\underline w_2,\underline w_3, \underline w_4] \\
    & = [H\underline w_1,H\underline w_2,H\underline w_3, H\underline w_4] \\
    & = [ \underline 0, \underline 1, H \underline w_3,
    \underline \infty] \\
    & = H\underline w_3
  \end{align*}
  Define $F = H^{-1}G$, then $F \underline z_i = \underline w_i$
  for each $i$.
\end{proof}

\begin{prop}
  If $\underline z_1,\underline z_2,\underline z_3,\underline z_4 \in
  \mathbb P^1(\mathbb A)$ are pairwise essentially distinct points then:
  \begin{enumerate}
    \item The four points are non-singular.
    \item $[\underline z_1,\underline z_2,\underline
  z_3,\underline z_4]$ is not $\underline 0, \underline 1,
  \underline \infty.$
  \item $[\underline z_1,\underline z_2,\underline
  z_3,\underline z_4]$ is not in one of the sets from Lemma
  \ref{sec:projective-geometry:infinitysets}.
  \end{enumerate}
\end{prop}

\begin{proof}
  For 1 and 3:
  As the $\underline z_i$ are pairwise essentially distinct, then none of the values of
  $x_iy_j -x_jy_i$ are $0$ or divide zero. So if $[\underline z_1,\underline z_2,\underline
  z_3,\underline z_4] = [u,v]$ then neither of $u,v$ divide zero or are zero.

  For 2:
  Let $A$ be the matrix such that $A\underline z_1 = \underline 0,
  A \underline z_2 = \underline 1, A \underline z_4 = \underline
  \infty$ then $[\underline z_1,\underline z_2,\underline
  z_3,\underline z_4] = A \underline z_3$. So if $[\underline
  z_1,\underline z_2,\underline z_3,\underline z_4] = \underline 0,
  \underline 1 \mbox{ or } \underline \infty$ then $ A \underline z_3
  = \underline 0, \underline 1 \mbox{ or } \underline
  \infty$. However $A$ is injective, hence a contradiction.
\end{proof}

\section{Permutations}
This section shows the dependence of the projective cross-ratio from cyclic permutations. The results
correspond to ~\cite{Beardon2005}*{Sec 13.4}.
\begin{rem}
  The matrices
  \begin{equation}\label{eq:2}
    \begin{pmatrix}
      1 & 0 \\
      0 & 1
    \end{pmatrix},
    \begin{pmatrix}
      0 & 1 \\
      1 & 0
    \end{pmatrix},
    \begin{pmatrix}
      -1 & 1 \\
      0 & 1
    \end{pmatrix},
    \begin{pmatrix}
      0 & 1 \\
      -1 & 1
    \end{pmatrix},
    \begin{pmatrix}
      1 & -1 \\
      1 & 0
    \end{pmatrix},
    \begin{pmatrix}
      1 & 0 \\
      1 & -1
    \end{pmatrix} \in PGL_2(\mathbb A)
  \end{equation} permutate $\underline 0, \underline 1, \underline \infty$ and correspond to the permutations
  \begin{equation}
    (\underline 0) (\underline 1) (\underline \infty),
    (\underline 1) (\underline 0 \mbox{ } \underline \infty),
    (\underline 0 \mbox{ } \underline 1) (\underline \infty),
    (\underline 0 \mbox{ } \underline 1 \mbox{ } \underline \infty),
    (\underline 1 \mbox{ } \underline 0 \mbox{ } \underline \infty),
    (\underline 0) (\underline 1 \mbox{ } \underline \infty),
  \end{equation}
  respectively.
\end{rem}
The members of \eqref{eq:2} are a group closed under matrix
multiplication. Each permutation of three points corresponds to a
matrix in $PGL_2(\mathbb A)$.
\begin{prop}\label{sec:permutations:prop:1}
  Given four non-singular pairwise distinct points $\underline z_i \in \mathbb
  P^1(\mathbb A)$, $i =1,2,3,4$ and a permutation $\rho \in S_4$. If $[\underline z_1,\underline
  z_2,\underline z_3, \underline z_4] = \underline \lambda$ for $\underline \lambda \in \mathbb P^1(\mathbb A)$, then
  \begin{equation}
    [\underline
    z_{\rho^{-1}(1)},\underline z_{\rho^{-1}(2)},\underline
    z_{\rho^{-1}(3)}, \underline z_{\rho^{-1}(4)}]= F_\rho \underline \lambda, \mbox{ for } F_\rho \in PGL_2(\mathbb A).
    \end{equation}
\end{prop}

\begin{proof}
  Let $\lambda = [\underline z_1,\underline
  z_2,\underline z_3, \underline z_4]$ and let $A \in GL_2(\mathbb A)$
  be the matrix such that $A \underline z_1 = \underline 0, A
  \underline z_2 = \underline 1, A \underline z_4 = \underline
  \infty.$ By the invariance of the projective cross-ratio under
  $PGL_2(\mathbb A)$, we see that $A \underline z_3 = \underline
  \lambda$. Now
  \begin{equation}
    [\underline z_{\rho^{-1}(1)},\underline z_{\rho^{-1}(2)},\underline z_{\rho^{-1}(3)}, \underline z_{\rho^{-1}(4)}] =
    [A \underline z_{\rho^{-1}(1)},A \underline z_{\rho^{-1}(2)},A \underline z_{\rho^{-1}(3)},A \underline z_{\rho^{-1}(4)}],
  \end{equation}
  is then the projective cross-ratio of $\underline 0, \underline
  1, \underline \lambda$ and $\underline \infty$ in some order. Hence
  it only relies on $\rho$ and $\underline \lambda$, so is of the form
  $F_\rho \underline \lambda$.
\end{proof}

\begin{prop}\label{sec:permutations:prop:2}
  The $F_\rho \in PGL_2(\mathbb A)$, for $\rho$ a transposition, is in
  \eqref{eq:2}.
\end{prop}

\begin{proof}
  The proof is just calculation and so for an illustration we shall do one example, say
  $\rho = (1 2), \lambda =
  \begin{bmatrix}
    u \\
    v
  \end{bmatrix}
  $ and a matrix $A$ from \eqref{eq:1} then
  \begin{align*}
    [\underline z_{\rho^{-1}(1)},\underline z_{\rho^{-1}(2)},\underline z_{\rho^{-1}(3)}, \underline z_{\rho^{-1}(4)}] & = [\underline z_2,\underline z_1,\underline z_3, \underline z_4] \\
    & = [A \underline z_{2},A \underline z_{1},A \underline z_{3},A \underline z_{4}] \\
    & = [\underline 1,\underline 0,\underline \lambda, \underline \infty] \\
    & =
    \begin{bmatrix}
      (v -u)(0-1) \\
      (1-0)(0-v)
    \end{bmatrix} \\
    & =
    \begin{pmatrix}
      -1 & 1 \\
      0 & 1
    \end{pmatrix} \lambda .
  \end{align*}

  Then for the others we have
  \begin{equation}
    \begin{gathered}
      F_{(12)} = F_{(34)} =
      \begin{pmatrix}
        -1 & 1 \\
        0 & 1
      \end{pmatrix}, \\
      F_{(23)} = F_{(14)} =
      \begin{pmatrix}
        0 & 1 \\
        1 & 0
      \end{pmatrix}, \\
      F_{(13)} = F_{(24)} =
      \begin{pmatrix}
        1 & 0 \\
        1 & -1
      \end{pmatrix}. \\
    \end{gathered}
  \end{equation}
\end{proof}

\begin{prop}\label{sec:permutations:prop:3}The matrices $F_\rho$ from \eqref{eq:2} have the property $ F_{\sigma \rho} =F_\sigma F_\rho,$ for $
  \sigma,\rho \in S_4.$
\end{prop}

\begin{proof}
  Let $\mu = \sigma \rho, \underline w_{\rho(k)} = \underline z_k
  \mbox{ and } \underline u_{\sigma(j)} = \underline w_j$ then
  $\underline u_{\mu(k)} =\underline u_{\sigma (\rho(k))} = \underline
  w_{\rho(k)} =\underline z_k$ hence:
  \begin{align*}
    F_\sigma F_\rho [\underline z_1,\underline z_2,\underline z_3, \underline z_4] & = F_\sigma [\underline w_1,\underline w_2,\underline w_3, \underline w_4] \\
    & = [\underline u_1,\underline u_2,\underline u_3, \underline u_4] \\
    & = [\underline z_{\mu^{-1}(1)},\underline z_{\mu^{-1}(2)},\underline z_{\mu^{-1}(3)}, \underline z_{\mu^{-1}(4)}] \\
    & = F_\mu [\underline z_1,\underline z_2,\underline z_3,
    \underline z_4]
  \end{align*}
\end{proof}

The previous propositions \ref{sec:permutations:prop:1},
\ref{sec:permutations:prop:2}, \ref{sec:permutations:prop:3} are
summarized in the following theorem. The theorem corresponds to
~\cite{Beardon2005}*{Thm 13.5.1}.

\begin{thm}
  For each $\rho \in S_4$ there is a matrix $F_\rho \in$
  \begin{equation*}
    \Gamma = \left\{ \begin{pmatrix}
        1 & 0 \\
        0 & 1
      \end{pmatrix},
      \begin{pmatrix}
        0 & 1 \\
        1 & 0
      \end{pmatrix},
      \begin{pmatrix}
        -1 & 1 \\
        0 & 1
      \end{pmatrix},
      \begin{pmatrix}
        0 & 1 \\
        -1 & 1
      \end{pmatrix},
      \begin{pmatrix}
        1 & -1 \\
        1 & 0
      \end{pmatrix},
      \begin{pmatrix}
        1 & 0 \\
        1 & -1
      \end{pmatrix} \right\},
  \end{equation*}
  which permute $\{ \underline 0, \underline 1, \underline \infty
  \}$, such that for any non-singular pairwise distinct $\underline z_1,\underline
  z_2,\underline z_3, \underline z_4$
  \begin{equation}
    [\underline z_{\rho^{-1}(1)},\underline z_{\rho^{-1}(2)},\underline z_{\rho^{-1}(3)}, \underline z_{\rho^{-1}(4)}] = F_\rho [\underline z_1,\underline z_2,\underline z_3, \underline z_4].
  \end{equation}
  So $\rho \mapsto F_\rho$ is a homomorphism of $S_4$ onto $\Gamma$
  with kernel \begin{equation}K = \{ I, (1 2)(3 4),(1 3)(2 4),(1 4)(2
    3) \}.\end{equation}
\end{thm}

\section{Cycles}

Cycles are natural objects invariant under M\"obius
transformations. They are defined by any of the following equivalent
equations, ~\cite{Kisil2a}, \cite{Kisil05a} \cite{Kisil06a}:

\begin{align*} Kx \bar x - L x \bar y -\bar L \bar x y +M y \bar y =0,\\
  \begin{pmatrix}
    \bar x & \bar y \\
  \end{pmatrix}
  \begin{pmatrix}
    K & \bar L \\
    - L & M
  \end{pmatrix}
  \begin{pmatrix}
    x \\
    y
  \end{pmatrix} = 0,\\
  \begin{pmatrix}
    - \bar y & \bar x
  \end{pmatrix}
  \begin{pmatrix}
    L & -M \\
    K & \bar L
  \end{pmatrix}
  \begin{pmatrix}
    x \\
    y
  \end{pmatrix} = 0.
\end{align*}
Where $K =k \iota, M = m \iota, k,m \in \mathbb R, L \in \mathbb
A$. Here is the projective version:
\begin{defn} A cycle $C$ is the set of points $\underline z \in
  \mathbb P^1(\mathbb A)$ satisfying
  \begin{equation}
    \underline {\bar z}^T
    \begin{pmatrix}
      K & \bar L \\
      - L & M
    \end{pmatrix} \underline z = 0,
  \end{equation} for $K =k \iota, M = m \iota, k,m \in \mathbb R, L \in \mathbb A$ with cycle matrix $\mathbf C =
  \begin{pmatrix}
    K & \bar L \\
    -L & M
  \end{pmatrix}
$.
\end{defn}

In Yaglom's book ~\cite{Yaglom1979}*{p261}, a cycle is defined to be the set of points satisfying
$[z_1,z_2,z,z_4] = [\bar z_1, \bar z_2, \bar z, \bar
z_4]$. This is the same as an equation $Az \bar z +Bz -\bar B \bar z
+C = 0$, hence the following proposition.

\begin{prop}\label{sec:cycles:crossratioformula}
  Let $ \underline z_1, \underline z_2, \underline z_4 \in \mathbb
  P^1(\mathbb A)$ be fixed pairwise essentially distinct points. Then the set of $\underline
  z \notin \{\underline z_1, \underline z_2, \underline z_4 \}$ satisfying $[\underline
  z_1,\underline z_2,\underline z, \underline z_4] = [\bar {\underline
    z}_1,\underline {\bar z}_2,\underline {\bar z}, \underline {\bar
    z}_4]$ with the points $\{\underline z_1, \underline z_2, \underline z_4 \}$, is a cycle.
\end{prop}

\begin{proof}
  $[\underline z_1,\underline z_2,\underline z, \underline z_4] =
  [\bar {\underline z}_1,\underline {\bar z}_2,\underline {\bar z},
  \underline {\bar z}_4]$ is the same as $A \underline z = \bar A \bar{
    \underline z}$ from Lemma \ref{rem 1}. Implying $\bar A^{-1}A
  \underline z = \underline{\bar z}$. As \begin{align*}
    \begin{bmatrix}
      -\bar y &  \bar x
    \end{bmatrix}
    \begin{bmatrix}
      \bar x \\
      \bar y
    \end{bmatrix} = (0), \\ \mbox{then}
    \begin{bmatrix}
      -\bar y &  \bar x
    \end{bmatrix} {\bar A}^{-1} A
    \begin{bmatrix}
      x \\
      y
    \end{bmatrix}
    = (0), \\
  \underline z^T
  \begin{pmatrix}
    0 & 1 \\
    -1 & 0
  \end{pmatrix} {\bar A}^{-1} A \underline z = 0.
\end{align*}
  Let:
  \begin{eqnarray*}
    L & = & (x_2y_1 -x_1y_2)(\bar x_2\bar y_3 -\bar x_3\bar y_2)\bar x_1 y_3 \\ & & \qquad - (\bar x_2\bar y_1 -\bar x_1\bar y_2)(x_2y_3 -x_3y_2)x_1 \bar y_3, \\
    K' & = & (\bar x_2\bar y_1 -\bar x_1\bar y_2)(x_2y_3 -x_3y_2)x_1 \bar x_3, \\
    M' & = & (\bar x_2\bar y_1 -\bar x_1\bar y_2)(x_2y_3 -x_3y_2) y_1 \bar y_3. \mbox{ Hence:} \\
    \begin{pmatrix}
    0 & 1 \\
    -1 & 0
  \end{pmatrix} \bar A^{-1} A & = & \frac{1}{det(\bar A)}
    \begin{pmatrix}
       (K'-\bar K') & \bar L\\
     -L & M' -\bar M'
    \end{pmatrix}, \\
  \end{eqnarray*}
  which is a cycle matrix.
\end{proof}

\begin{cor}
  Any three pairwise essentially distinct points $\underline z_1, \underline z_2$ and $\underline z_4$ define a cycle, via Proposition \ref{sec:cycles:crossratioformula}. In particular, the $\underline z_i$ are in that cycle.
\end{cor}

\begin{ex}\label{sec:cycles:ex1}
  If cycle contains $\underline 0, \underline a$ and $\underline
  \infty, a \in \mathbb R \backslash \{0\}$, then it is the real
  line. As
  \begin{eqnarray*}
    \begin{bmatrix}
      0 & 1
    \end{bmatrix}
    \begin{pmatrix}
      K & \bar L \\
      -L & M
    \end{pmatrix}
    \begin{bmatrix}
      0 \\
      1
    \end{bmatrix}&=&M = 0, \\
    \begin{bmatrix}
      1 & 0
    \end{bmatrix}
    \begin{pmatrix}
      K & \bar L \\
      -L & M
    \end{pmatrix}
    \begin{bmatrix}
      1 \\
      0
    \end{bmatrix}&=&K = 0, \\
    \begin{bmatrix}
      a & 1
    \end{bmatrix}
    \begin{pmatrix}
      0 & \bar L \\
      -L & 0
    \end{pmatrix}
    \begin{bmatrix}
      a \\
      1
    \end{bmatrix}&=&a(\bar L - L) = 0, \\
  \end{eqnarray*} $L \in \mathbb R$. Cycle matrix
  $\begin{pmatrix}
    0 & b \\
    -b & 0
  \end{pmatrix}$, for $b \in \mathbb R \backslash \{0\}$, is then the
  matrix representing the real line. Similarly if $\underline 0,
  \underline a\iota$ and $\underline \infty,$ for $a \in \mathbb R
  \backslash \{0\}$, are in a cycle then it is the imaginary axis, with
  cycle matrix $
  \begin{pmatrix}
    0 & -b \iota \\
    -b \iota & 0 \\
  \end{pmatrix}, b \in \mathbb R \backslash \{0\}$.
\end{ex}

\begin{rem} From the known property of determinant
  \begin{align*}
    det(\bar A^{-1}  A) & = det( \bar A^{-1})det(A) \\
    & = \frac{(x_2y_1 - x_1y_2)(x_2y_4 - x_4y_2)(x_1y_4-x_4y_1)}{(\bar
      x_2\bar y_1 - \bar x_1\bar y_2)(\bar x_2\bar y_4 - \bar x_4\bar
      y_2)(\bar x_1\bar y_4-\bar x_4\bar y_1)} = |L|^2 + KM.
  \end{align*}
\end{rem}

\begin{defn}
  A set of distinct points $\{ \underline z_1,\underline
  z_2, ... ,\underline z_n \} \subset \mathbb P^1(\mathbb A)$ are
  concyclic if and only if they satisfy $\underline {\bar z_k} \mathbf C
  \underline z_k =0$ for the same cycle matrix $\mathbf C$.
\end{defn}

This definition provides the following theorem related to the results
~\cite{Yaglom1979}*{p275. (42a)}, ~\cite{Beardon2005}*{Thm 13.4.4}.

\begin{cor}
  Four pairwise distinct points $z_1,z_2,z_3,z_4 \in \mathbb A,$ are concyclic in
  $\mathbb A$ if and only if $\mathfrak S(z_i)$ for all $i =1,2,3,4$
  are also concyclic.
\end{cor}

\begin{proof}
  Four distinct points in $\mathbb A$ are concylic if and only if the original
  cross-ratio of them is real, from
  \cite{Yaglom1979}*{p275,(42a)}. Let $\underline z_i = \mathfrak
  S(z_i)$ for each $i$, and $\underline z_1,\underline z_2,\underline
  z_3, \underline z_4$ be concyclic in $\mathbb P^1(\mathbb A)$. Then let $\underline \lambda =
  [\underline z_1,\underline z_2,\underline z_3, \underline
  z_4]$. Since $\underline \lambda = \underline{\bar \lambda}$ then
  $\mathfrak P(\underline \lambda) \in \mathbb R$. Hence due to
  \ref{sec:cross-ratio:rem:4} the original cross-ratio is real. So the
  points $z_i$ are concyclic. The necessity follows similarly. If $z_i \in \mathbb A$ are concyclic, $\mathfrak S(z_i)$ are also concyclic.
\end{proof}

\begin{defn}\label{newcycle}
  A cycle $C$, with cycle matrix $\mathbf C$, is mapped by a matrix $A \in PGL_2(\mathbb A)$, to the cycle $C'$ with cycle matrix $\mathbf{ C'} = \bar A^T \mathbf C A$.
\end{defn}

This leads to a proposition about points on two
interconnecting cycles \cite{Beardon2005}*{Exercise 13.4}, which only applies to $\mathbb O$ and $\mathbb C$.
\begin{prop}\label{sec:cycles:orthogonality}
  For $\mathbb A = \mathbb O$ or $\mathbb C$:
  Given two distinct intersecting cycles $C$, $C'$ and four pairwise
  essentially distinct points $\underline z_1,\underline
  z_2, ,\underline z_3, \underline z_4 \in \mathbb P^1(\mathbb A)$, such that $\underline z_1,\underline
  z_2, \underline z_4 \in C$ and $\underline z_1, \underline z_3,
  \underline z_4 \in C'$.  The $C$ and $C'$ can be mapped to the
  real and imaginary axes by a matrix $A \in PGL_2(\mathbb A)$ if and
  only if $[\underline z_1,\underline z_2,\underline
  z_3, \underline z_4]$ is on the imaginary axis.
\end{prop}

\begin{proof}
  Say $[\underline z_1,\underline z_2,\underline
  z_3, \underline z_4]$ is on the imaginary axis, then let $A$ be the matrix from \eqref{eq:1} with $\underline z_3$
  replaced by $\underline z_4$. Say $A$ has mapped $C$ and $C'$ to $R$ and $I$ respectively. Notice $A \underline z_1 = \underline 0, A \underline z_4 = \underline \infty \in R \cup I$, $A \underline z_2 =\underline 1 \in R$ and $[\underline z_1,\underline z_2,\underline
  z_3, \underline z_4] = A \underline z_3 = [t\iota,1] \in I,$ for $ t \in \mathbb R$. Then due to Example \eqref{sec:cycles:ex1} we conclude $R$ is the real axis and $I$ is the imaginary axis.

 Say there exists a matrix $A \in PGL_2(\mathbb A)$, such that it maps $C$ to the real axis and $C'$ to the imaginary axis. As $\underline z_1$ and $\underline z_4$ are on $C$ and $C'$, $A \underline z_1$ and  $A \underline z_4$ are both on the real and imaginary axis. So $ A \underline z_1, A \underline z_4 \in \{\underline 0,\underline \infty\}$. Say $A \underline z_2 = [r,1]$ and $A \underline z_3 = [t \iota,1], r,t \in \mathbb R$, then:
 \begin{align*}
   [\underline z_1,\underline z_2,\underline
  z_3, \underline z_4] & =  [A \underline z_1,A \underline z_2,A \underline
  z_3, A \underline z_4], \\
  & =  \left\{
  \begin{array}{l l}
    \mbox{if } A \underline z_1 = \underline 0 & [\underline 0,[r,1],[t \iota,1], \underline \infty] = [t \iota,1], \\
    \mbox{if } A \underline z_4 = \underline 0 & [\underline \infty,[r,1],[t \iota,1], \underline 0] = [1,t \iota] = [\frac{1}{t} \iota, \iota^2]. \\
  \end{array} \right.
 \end{align*}
 For both values of $A \underline z_1$, $[\underline z_1,\underline z_2,\underline
  z_3, \underline z_4]$ is then on the imaginary axis.
\end{proof}

This suggests the following definition:

\begin{defn}
  For $\mathbb A = \mathbb O$ or $\mathbb C$:
  Two cycles $C, C'$ are projective-orthogonal if and only if there exists four pairwise distinct points $\underline z_1,\underline
  z_2,\underline z_3, \underline z_4$ such that $\underline z_1, \underline z_4 \in C \cap C'$, $\underline z_2 \in C, \underline z_3 \in C'$ and $[\underline z_1,\underline z_2,\underline
  z_3, \underline z_4]$ is on the imaginary axis.
\end{defn}

This is then similar to cycle-orthogonality defined in
\cite{Kisil2a}*{Sec 5.3} \cite{Kisil05a}, given as:
\begin{defn} The cycle product of two cycles $C, C'$ with cycle matrices $\mathbf C, \mathbf{C'}$ is
  \begin{equation}
    \langle \mathbf C,\mathbf{C'}\rangle = -tr\left(
    \begin{pmatrix}
      0 & -1 \\
      1 & 0
    \end{pmatrix}
\mathbf C \begin{pmatrix}
      0 & -1 \\
      1 & 0
    \end{pmatrix}\mathbf{\bar C'}\right),
  \end{equation} where the r.h.s. is the product of matrices. The two cycles are then cycle-orthogonal if $\langle \mathbf C, \mathbf{C'} \rangle = 0$.
\end{defn}

\begin{lem}\label{cycleidentity}
  For all $A \in PGL_2(\mathbb A)$ and cycles $C$ and $C'$,
  \begin{equation}
    \langle \mathbf C,\mathbf{C'}\rangle = \det(A)^2\langle \bar A^T\mathbf CA,\bar A^T\mathbf{C'}A\rangle.
  \end{equation}
\end{lem}

\begin{proof}
  \begin{align*}
    \langle \bar A^T\mathbf CA,\bar A^T\mathbf{C'}A\rangle & = -tr\left(
    \begin{pmatrix}
      0 & -1 \\
      1 & 0
    \end{pmatrix}
\bar A^T\mathbf CA \begin{pmatrix}
      0 & -1 \\
      1 & 0
    \end{pmatrix}A^T\mathbf{\bar C'}\bar A\right), \\
  & = -tr \left( \det(A)^2
    \begin{pmatrix}
      0 & -1 \\
      1 & 0
    \end{pmatrix}
    \begin{pmatrix}
      0 &1 \\
      -1 & 0
    \end{pmatrix}\bar A^{-1} \begin{pmatrix}
      0 & -1 \\
      1 & 0
    \end{pmatrix} \mathbf C A \right.\\
    & \left. \begin{pmatrix}
      0 & -1 \\
      1 & 0
    \end{pmatrix}
    \begin{pmatrix}
      0 &1 \\
      -1 & 0
    \end{pmatrix} A^{-1} \begin{pmatrix}
      0 & -1 \\
      1 & 0
    \end{pmatrix} \mathbf{\bar C'} \bar A \right), \\
  & = -tr \left( \det(A)^2
    \bar A^{-1} \begin{pmatrix}
      0 & -1 \\
      1 & 0
    \end{pmatrix} \mathbf C A A^{-1} \begin{pmatrix}
      0 & -1 \\
      1 & 0
    \end{pmatrix} \mathbf{\bar C'} \bar A \right) \\
  & = \det(A)^2 \langle \mathbf C,\mathbf{C'}\rangle. \\
  \end{align*}
\end{proof}

We can then show the equivalence of cycle-orthogonal and projective-orthogonal.

\begin{thm}
  For $\mathbb A = \mathbb O$ or $\mathbb C$:
  Two distinct cycles $C,C'$ are projective-orthogonal if and only if they are
  cycle-orthogonal.
\end{thm}

\begin{proof}
  Say you can map $C$ to the real line defined by the cycle matrix $\mathbf R =
  \begin{pmatrix}
    0 & l \\
    -l & 0
  \end{pmatrix}$ for some $l \in \mathbb R$, and $C'$ to the imaginary
  line defined by the cycle matrix $\mathbf I =
  \begin{pmatrix}
    0 & r \iota \\
     r\iota & 0
  \end{pmatrix}, r \in \mathbb R$ both by the same matrix $B$. Then by Lemma \ref{cycleidentity}
  \begin{align*}
    \langle \mathbf C,\mathbf{C'} \rangle
    & = -\det(B)^2tr(\begin{pmatrix}
      0 & -1 \\
      1 & 0
    \end{pmatrix}\mathbf R\begin{pmatrix}
      0 & -1 \\
      1 & 0
    \end{pmatrix}\bar{\mathbf I})\\
    & = \det(B)^2(-lr \iota +lr \iota) = 0.
  \end{align*}

  Say $\langle \mathbf C , \mathbf C' \rangle = 0$, and $\underline z_1, \underline z_2, \underline z_3 \in C$ are distinct points. Let $A \in PGL_2(\mathbb A)$ be the matrix such that $A\underline z_1 = \underline 0, A\underline z_2 = \underline 1$ and $ A \underline z_3 = \underline \infty$, then $\bar A^T\mathbf C A = \mathbf R$. Let $\bar A^T \mathbf C' A =
  \begin{pmatrix}
    K' & \bar L' \\
    -L' & M'
  \end{pmatrix}$, then by Lemma \ref{cycleidentity}:
  \begin{align*}
    \langle \mathbf C , \mathbf C' \rangle & = -det(A)^2tr \left(
      \begin{pmatrix}
        0 & -1 \\
        1 & 0
      \end{pmatrix} \mathbf R
      \begin{pmatrix}
        0 & -1 \\
        1 & 0
      \end{pmatrix}
      \begin{pmatrix}
        \bar K' & L' \\
        -\bar L' & \bar M'
      \end{pmatrix} \right), \\
  & = -det(A)^2tr \left(
    \begin{pmatrix}
      l & 0 \\
      0 & l
    \end{pmatrix}
    \begin{pmatrix}
      \bar L' & -\bar M' \\
      \bar K' & L'
    \end{pmatrix} \right), \\
& = -det(A)^2tr \left(
    \begin{pmatrix}
      l\bar L' & -l\bar M' \\
      l\bar K' & lL'
    \end{pmatrix} \right), \\
  & = -det(A)^2(l \bar L' +Ll) =0. \\
  \end{align*}
This implies $L' = - \bar L$, hence $L = r \iota, r \in \mathbb R$. So $\bar A^T \mathbf C' A = \mathbf I$.
\end{proof}

\section{Lobachevskian Geometry}

Lobachevskian geometry (Hyperbolic Geometry) is defined in the upper half plane --- $\mathbb
H = \{ x + y \iota, x,y \in \mathbb R, y > 0 \}$. It is a non-Euclidean
geometry, where for every point $P$ and line $L$ not intersecting $P$, there are an infinite
number of lines through $P$ which never cross $L$,
\cite{Beardon2005}*{Section 5.1}. The Lobachevskian distance on
$\mathbb H$ is defined as a metric which is invariant under M\"obius transformations.


\begin{defn}\cite{Beardon2005}*{Chapter 14}
  Given $z,w \in \mathbb H$, take $C$ to be the circle perpendicular
  to the real line such that $z,w \in C$. Let $u, v \in \mathbb R$ be
  the points where $C$ intersect the real line. The Lobachevskian
  distance $\rho(z,w)$ from $z$ to $w$ is defined as
  \begin{equation}
    \rho(z,w) = \ln ( [u,z,w,v] ),
  \end{equation} from the original Cross-ratio \ref{crossratio}.
\end{defn}

Notice that $[u,z,w,v]$ is real as $u,z,v$ and $w$ are concyclic. The
following expressions are linked to the Lobachevskian distance
\cite{Beardon1983}*{Thm 7.2.1}.

\begin{thm}\label{sec:lobach-geom:1}For any $z,w \in \mathbb H$ the
  Lobachevskian distance satisfy:
  \begin{align*}
    \rho( z, w) & = \ln \left | \frac{| z - {\bar w}| + | z-  w|}{| z - {\bar w}| - | z-  w|} \right |, \\
    \cosh^2 \left(\frac{\rho( z, w)}{2} \right) & = \frac{| z - {\bar w}|^2}{4  {\Im}( z)  {\Im}( w)}, \\
    \sinh^2 \left(\frac{\rho( z, w)}{2} \right) & = \frac{| z - {w}|^2}{4  {\Im}( z)  {\Im}( w)}, \\
    \tanh^2 \left(\frac{\rho( z, w)}{2} \right) & = \frac{| z - {w}|^2}{| z - {\bar w}|^2}. \\
  \end{align*}
\end{thm}
The property $ \tanh^2(\rho(z,w)/2) = |z-w|^2/|z - \bar w|^2 = [w,\bar z, z,
\bar w]$ is used to define the projective Lobachevskian distance.

\begin{defn}
  Let $\underline z, \underline w \in \mathbb P^1(\mathbb A)$ be essentially distinct points such
  that $\underline z$ and $\underline {\bar w}$ are also essentially
  distinct.  The $\tanh$-projective Lobachevskian distance $d$ between
  $\underline z, \underline w$ is defined as
  \begin{equation}
    d(\underline z,\underline w) = \mathfrak P([\underline w,\underline {\bar z},\underline z,\underline{\bar w}]).
  \end{equation} The projective Lobachevskian distance $\delta(\underline z,\underline w)$ is then defined by
  \begin{equation}
    \tanh^2(\delta (\underline z,\underline w) / 2) = d(\underline z,\underline w)
  \end{equation}
\end{defn}

Note that $d(\underline z,\underline w) = r \in \mathbb R$ as
$[\underline {\bar w},\underline {z},\underline {\bar
  z},\underline{w}] = F_{(1 \mbox{ } 4)(2 \mbox{ } 3)}[\underline
w,\underline {\bar z},\underline z,\underline{\bar w}] = [\underline
w,\underline {\bar z},\underline z,\underline{\bar w}]$.
\begin{rem}Given $\underline z, \underline w \in \mathbb P^1(\mathbb
  A)$, there are the following relations for the $\tanh$-Lobachevskian metric:
  \begin{enumerate}
  \item It is invariant under conjugation,
    $d(\underline z,\underline w) = d(\underline {\bar z},\underline
    {\bar w})$.
  \item It's symmetric, $d(\underline z,\underline w) = d(\underline
    w,\underline z)$.
  \item $d(\underline {\bar z},\underline w) = 1/d(\underline
    z,\underline w)$.
  \end{enumerate}
\end{rem}

There is an equivalent to Theorem \ref{sec:lobach-geom:1} with the
projective Lobachevskian distance.
\begin{thm} For $\underline z, \underline w \in \mathbb P^1(\mathbb
  A)$
  \begin{align*}
    \delta(\underline z,\underline w) & = \ln \left(\frac{d(\underline z,\underline w)^{\frac{1}{2}} +1}{1- d(\underline z,\underline w)^{\frac{1}{2}}} \right), \\
    \cosh^2 \left(\frac{\delta( \underline z,\underline w)}{2} \right) & = \frac{1}{1-d(\underline z,\underline w)} ,\\
    \sinh^2 \left(\frac{\delta( \underline z,\underline w)}{2} \right) & = \frac{d(\underline z,\underline w)^{\frac{1}{2}}}{1-d(\underline z,\underline w)} .\\
  \end{align*}
\end{thm}

\begin{proof}
  As $\tanh(x/2) = (e^x -1)/(e^x+1)$, the values can be quickly calculated.

  Then from Remark \ref{sec:cross-ratio:rem:4} we know $d(\underline
  z,\underline w) = |z-w|^2/|z-\bar{w}|$ for $\mathfrak S(z) =
  \underline z, \mathfrak S(w) = \underline w$. Hence we can show the
  similarity with Theorem \ref{sec:lobach-geom:1} by:
  \begin{equation}
    \delta(\underline z,\underline w) =\ln \left( \frac{d(\underline z,\underline w)^{\frac{1}{2}} +1}{1- d(\underline z,\underline w)^{\frac{1}{2}}} \right)= \ln \left( \frac{\frac{|z-w|}{|z- \bar w|} +1}{1-\frac{|z-w|}{|z- \bar w|}} \right) = \rho(\underline z,\underline w).
  \end{equation}
  Similar calculations give the values for $\sinh(\delta/2)$ and $\cosh(\delta/2)$. \end{proof}
The Lobachevskian distance is invariant under $PGL_2(\mathbb
R)$ as shown in the next theorem.
\begin{thm}
  For all $A \in PGL_2(\mathbb R)$, and $\underline z,\underline w \in \mathbb P^1(\mathbb A)$ there exists the following identity: $\delta(A\underline z,A\underline w) = \delta(\underline z,\underline w)$.
\end{thm}

\begin{proof}We calculate:
  \begin{align*}
    \delta(A \underline z,A \underline w) &
    = \ln \left(\frac{d(A \underline z,A \underline w)^{\frac{1}{2}} +1}{1-d(A \underline z,A \underline w)^{\frac{1}{2}}} \right) \\
    & = \ln \left(\frac{\mathfrak P([A \underline w,\bar A \underline {\bar z},A \underline z,\bar A \underline{\bar w}])^{\frac{1}{2}} +1}{1- \mathfrak P([A \underline w,\bar A \underline {\bar z},A \underline z,\bar A \underline{\bar w}])^{\frac{1}{2}}} \right) \\
    & = \ln \left(\frac{\mathfrak P([A \underline w, A \underline {\bar z},A \underline z, A \underline{\bar w}])^{\frac{1}{2}} +1}{1- \mathfrak P([A \underline w, A \underline {\bar z},A \underline z, A \underline{\bar w}])^{\frac{1}{2}}} \right) \\
    & = \ln \left(\frac{d(\underline z,\underline w)^{\frac{1}{2}} +1}{1- d(\underline z,\underline w)^{\frac{1}{2}}} \right) = \delta( \underline z, \underline w).\\
  \end{align*} Here we used that $A$ is a real matrix.
\end{proof}

\begin{prop}
  For $\underline z,
  \underline w \in \mathbb P^1(\mathbb A)$ such that $\mathfrak
  P(\underline z), \mathfrak P(\underline w) \in \mathbb H$, $\delta(\underline z,\underline w) = \rho(\mathfrak
  P(\underline z),\mathfrak P(\underline w))$.
\end{prop}

\begin{proof}
  Take $\underline z = (x, y), \underline w = (u,v)$ then
  let $z = \mathfrak P(\underline z) = x/y, w = \mathfrak P(\underline
  w) = u/v$. Then
  \begin{align*}
    \tanh^2(\rho(z,w)/2)  & = \frac{| z - {w}|^2}{| z - {\bar w}|^2} = \frac{| x/y - {u/v}|^2}{| x/y - {\bar u/ \bar v}|^2} = \frac{| xv - {uy}|^2/|vy|^2}{| x \bar v - {\bar u y}|^2/|\bar v y|^2} \\
    & = \mathfrak P \begin{pmatrix}
      |xv-uy|^2 \\
      |\bar x v-u \bar y|^2
    \end{pmatrix} = \mathfrak P([\underline w,\underline {\bar z},\underline z,\underline{\bar w}]) \\
    & = d(\underline z,\underline w) = \tanh^2(\delta(\underline z,
    \underline w)/2).
  \end{align*}
\end{proof}

Acknowledgement: This was partially supported by Nuffield foundation
summer bursary, Supervised by Vladimir V. Kisil. The author is grateful to the anonymous referee for numerous suggestions. Final version found at \url{SpringerLink.com}.
\bibliographystyle{plain} \bibliography{reference}

\end{document}